\documentclass[11pt]{article}

\usepackage[backend=bibtex,maxbibnames=9,firstinits=true,doi=false,isbn=false,url=false]{biblatex}
\DeclareNameAlias{sortname}{last-first}
\renewbibmacro{in:}{}

\usepackage{fullpage}
\usepackage[T1]{fontenc}
\usepackage{hyperref}
\usepackage{xcolor}
\definecolor{link-color}{rgb}{0.15,0.4,0.15}
\hypersetup{
  colorlinks,
  linkcolor = {link-color}, citecolor = {link-color},
}
\usepackage{graphicx}
\usepackage[font=small,labelfont=bf]{caption}
\usepackage{subcaption}
\usepackage{amsmath,amsthm,amssymb}
\usepackage{bbm}

\newtheorem{theorem}{Theorem}[section]

\newtheorem{lemma}[theorem]{Lemma}
\theoremstyle{plain}
\newtheorem{remark}[theorem]{Remark}
\newtheorem{definition}[theorem]{Definition}

\newcommand{\N}{\mathbb{N}}
\renewcommand{\P}{\mathbb{P}}
\newcommand{\E}{\mathbb{E}}
\newcommand{\1}{\mathbbm{1}}

\newcommand{\eps}{\varepsilon}

\title{Asymptotic number of caterpillars of regularly varying \texorpdfstring{$\Lambda$}{Lambda}-coalescents that come down from infinity}
\author{Bat\i{} \c{S}eng\"ul}

\bibliography{references.bib}

\begin{document}
\maketitle
\begin{abstract}
  In this paper we look at the asymptotic number of $r$-caterpillars for $\Lambda$-coalescents which come down from infinity, under a regularly varying assumption.
  An $r$-caterpillar is a functional of the coalescent process started from $n$ individuals which, roughly speaking, is a block of the coalescent at some time, formed by one line of descend to which $r-1$ singletons have merged one by one.
  We show that the number of $r$-caterpillars, suitably scaled, converge to an explicit constant as the sample size $n$ goes to $\infty$.
\end{abstract}
\section{Introduction and results}
A coalescent process is a particle system in which particles merge into blocks.
Coalescent processes have found a variety of applications in physics, chemistry and most notably in genetics where the coalescent process models ancestral relationships as time runs backwards.
The work on coalescent theory dates back to the seminal paper~\cite{kingman} where Kingman considered coalescent processes with pairwise mergers.
This was extended by \textcite{pitmanlambdacoal}, \textcite{sagitovlambdacoal} and \textcite{dk99}, to the case where multiple mergers are allowed to happen.

Let $\Lambda$ be a finite measure on $[0,1]$.
The $\Lambda$-coalescent $\Pi=(\Pi(t):t \geq 0)$ is a Markov process which takes values in the set of partitions of $\N$, which starts from $(\{1\},\dots)$ and evolves forwards in time by merging together several blocks into one block.
Such processes are characterised by the rates $\lambda_{b,k}$ at which $k$ fixed blocks coalesce into one block when the current state has $b$ blocks in total, that are given by
\[
  \lambda_{b,k}=\int_0^1p^{k-2}(1-p)^{b-k}\,\Lambda({\rm d}p).
\]
We refer to \textcite{berestyckibook} and \textcite{bertoinfragment} for an overview of the field.

A finite measure $\Lambda$ is said to be strongly regularly varying, SRV$(\alpha)$, with index $\alpha \in (0,2)$ if $\Lambda(dp)=f(p)\,dp$ and there exists a constant $A_\Lambda>0$ such that
\begin{equation}\label{eq:regular_vary}
  \lim_{p\downarrow 0}\frac{f(p)}{p^{1-\alpha}}= A_\Lambda.
\end{equation}
We extend the definition to include $\alpha=2$ by saying that $\Lambda$ is SRV$(2)$ when $\Lambda = \delta_{\{0\}}$.
$\Lambda$-coalescents, when $\Lambda$ is SRV$(\alpha)$, cover many important classes of coalescent processes such as Kingman's coalescent ($\alpha=2$), Bolthausen-Sznitman coalescent and Beta$(2-\alpha,\alpha)$-coalescents.
In this paper, we shall additionally restrict ourselves to the case when $\alpha\in(1,2]$ which is the case when the coalescent comes down from infinity.

\begin{figure}
  \centering
  \begin{subfigure}[b]{0.18\textwidth}
    \includegraphics[width=\textwidth]{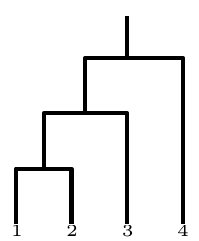}
  \end{subfigure}%
  \hspace{90pt}
  \begin{subfigure}[b]{0.18\textwidth}
    \includegraphics[width=\textwidth]{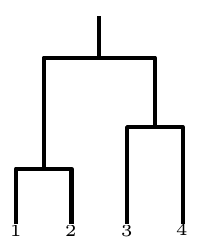}
  \end{subfigure}
  \caption{
  On the left is a coalescent tree with one $2$-caterpillar $\{1,2\}$, one $3$-caterpillar $\{1,2,3\}$ and one $4$-caterpillar $\{1,2,3,4\}$.
  The picture on the right consists of two $2$-caterpillars, $\{1,2\}$ and $\{3,4\}$, note however that $\{1,2,3,4\}$ in this picture is not a $4$-caterpillar.
}\label{fig:pronged_caterpillar}
\end{figure}

In evolutionary biology, an important task is to determine which coalescent process underlies a given data set.
To do this, it is useful to compute functionals of coalescent processes that are easy to check against a data set.
In this paper we study the functional known in the biology literature as $r$-caterpillars (in the case $r=2$, this is sometimes referred to as cherries).  
Loosely speaking, an $r$-caterpillar is a block of the coalescent at some time, formed by one line of descend to which $r-1$ singletons have merged with one by one, see Figure~\ref{fig:pronged_caterpillar} for an illustration.
To make this definition rigorous, we first introduce some notation.
For $n\in \N$, let $\Pi^{(n)}$ be the restriction of $\Pi$ to $\{1,\dots,n\}$.
We order the blocks of a partition by infimum and for $i \leq n$ and $t \geq 0$, let $c_t(i)$ be the number of the block of $\Pi^{(n)}(t)$ which contains $i$, so that for every $t \geq 0$, $i\in \Pi^{(n)}_{c_t(i)}(t)$.

\begin{definition}
  For $r \in \{1,\dots,n\}$, a set $B \subset \{1,\dots,n\}$ is called an $r$-caterpillar if $|B|=r$ and there exists a $t\geq 0$ such that
  \begin{itemize}
    \item $B$ is a block of $\Pi^{(n)}(t)$,
    \item there exists an $i\in B$ such that the function $s \mapsto |\Pi^{(n)}_{c_s(i)}(s)|$, for $s \in [0,t]$, has jumps of size one.
  \end{itemize}
\end{definition}

In the case $r=1$, the $1$-caterpillars are precisely $\{1\},\dots,\{n\}$.
Notice that the number of $r$-caterpillars only depends on the shape of the coalescent and are invariant under time-changes.

The main result of this paper gives asymptotic number of $r$-caterpillars of SRV$(\alpha)$ coalescent processes, as $n$ tends to $\infty$.

\begin{theorem}\label{thm:cherry_conv}
  Let $\Lambda$ be a finite SRV$(\alpha)$ measure with $\alpha\in (1,2]$ and let $\Pi^{(n)}$ be the restriction to $\{1,\dots,n\}$ of the $\Lambda$-coalescent $\Pi$.
  For $r \in \{2,\dots\}$ let $\xi_r^{(n)}$ denote the number of $r$-caterpillars associated to $\Pi^{(n)}$, then almost surely
      \[
        \lim_{n \rightarrow \infty} \frac{1}{n}\xi_r^{(n)} = \frac{\alpha^{r-1}}{2} \frac{\Gamma\left(1+\frac{\alpha}{\alpha-1}\right)}{\Gamma\left(r+\frac{\alpha}{\alpha-1}\right)}.
      \]
\end{theorem}

In the case of Kingman's coalescent ($\alpha=2$), Theorem~\ref{thm:cherry_conv} states that $\lim_{n \rightarrow \infty} \xi_r^{(n)}/n = 2^{r-1}/{(r+1)!}$ almost surely.
This agrees with the results in the literature~\cite{MR1747204,bio_exact,rosenberg_mean} where exact formulas of the expectation and variance for finite $n$ are known.
In the case of the Beta-coalescents, several related statistics have appeared in the literature, see for example~\cite{MR3224287,MR3027896,MR3025690,miller2016hydrodynamic,MR3466919}.

Let us briefly discuss the case when the index of regular variation lies in $(0,1]$.
In the case when $\alpha=1$ we suspect that similar arguments in this paper can be used to show that
\[
  \lim_{n\to\infty} \frac{(\log n)^r}{n}\xi_r^{(n)} = \frac{1}{r(r-1)(r-2)}
\]
almost surely.
The case when $\alpha\in (0,1)$ our methods fail because the limiting objects are no longer deterministic.
In this case we suspect that the number of $r$-caterpillars, when properly scaled, converge to an exponential integral of a subordinator and in future work we hope to explore this.

\section{Outline of the proof and the paper}\label{sec:outline}

We will reveal the $r$-caterpillars associated to $\Pi^{(n)}$ by exploring these thought time.
A caterpillar seen \emph{up to} time $t$ is a caterpillar $B\subset \{1,\dots,n\}$ which appears as a block of $\Pi^{(n)}(s)$ for some $s\leq t$.
The number of $r$-caterpillars seen up to time $t$ is increasing in $t$ and converges to $\xi^{(n)}_r$ as $t\to\infty$.
A caterpillar seen \emph{at} time $t$ is a caterpillar $B$ which is a block of $\Pi^{(n)}(t)$.
The number of $r$-caterpillars up to time $t$ increases, if a singleton ($1$-caterpillar) at time $t$ merges with an $(r-1)$-caterpillar at time $t$.
We look at a process which records the number of $\ell$-caterpillars at time $t$, for all $\ell\leq r-1$, and show that when suitably scaled, this process converges to the solution of a series of simultaneous ODEs (which we can solve).
After establishing this convergence, we use a simple argument to then show the convergence of the number of $r$-caterpillars up to time $t$ and then take $t\uparrow \infty$ to show Theorem~\ref{thm:cherries_at}.

The paper is organised as follows.
Then in Section~\ref{sec:rates} we use the regularly varying assumption to prove some lemmas about  the rate of mergers.
In Section~\ref{sec:at}, using the estimates we have obtained in the previous section, we show an auxiliary theorem about the convergence of the number of caterpillars at height $t$.
Finally in Section~\ref{sec:proof} we prove Theorem~\ref{thm:cherry_conv} by using the auxiliary theorem.

\section{Estimates on the rates}\label{sec:rates}

In this section we provide some estimates on the rates which will prove useful throughout the paper.
The limiting results for various rates have appeared in the literature, for example in~\cite[equation(10)]{berestyckismall},~\cite[Lemma 4]{bertoinflows3}.
The aim of this section is to obtain these convergences in a uniform way.

Thoughout this section suppose that $\Lambda$ is a finite SRV$(\alpha)$ measure with $\alpha\in (0,2)$.
Although later on we only use the case when $\alpha\in (1,2)$, we nevertheless show the identities in generality.
Note that when $\alpha=2$, we have that $\lambda_{b,k}=\1_{\{k=2\}}$ and similar results to those given here follow easily.

We begin with the following estimate on the rates.

\begin{lemma}\label{lemma:zeta_convergence}
  Suppose that $\alpha \in (0,2)$. Then for each $\eps>0$, there exists a $p\in(0,1)$ such that for each $k \in \{2,\dots,b\}$,
  \[
    \left|\binom{b}{k} \frac{\lambda_{b,k}}{b^\alpha}- A_\Lambda\frac{\Gamma(k-\alpha)}{\Gamma(k+1)}\right| \leq C(\eps+ b^{-1})\frac{\Gamma(k-\alpha)}{\Gamma(k+1)}+ C \frac{p^{b-k-2}}{b^\alpha},
  \]
  where the constant $C>0$ depends only on the measure $\Lambda$. 
\end{lemma}
\begin{proof}
  Fix $\alpha\in (0,2)$ and $\eps>0$. It follows from simple computations (see equation (23) in~\cite{tangent_cones}) that there exists a $p\in (0,1)$ such that
  \[
    \left|\binom{b}{k} \lambda_{b,k}- A_\Lambda\frac{\Gamma(k-\alpha)\Gamma(b+1)}{\Gamma(k+1)\Gamma(b+1-\alpha)}\right| \leq \eps A_\Lambda\frac{\Gamma(k-\alpha)\Gamma(b+1)}{\Gamma(k+1)\Gamma(b+1-\alpha)}+ p^{b-k-1}(\Lambda[0,1]+A_\Lambda p^{1-\alpha}).
  \]
  On the other hand by Striling's approximation there exists a constant $C>0$ such that
  \[
    \left| b^{-\alpha} \frac{\Gamma(b+1)}{\Gamma(b+1-\alpha)} -1\right | \leq C b^{-1}.
  \]
  Thus by the triangle inequality
  \begin{align*}
    \left|\binom{b}{k} \frac{\lambda_{b,k}}{b^\alpha}- A_\Lambda\frac{\Gamma(k-\alpha)}{\Gamma(k+1)}\right| & \leq  b^{-\alpha}\left|\binom{b}{k} \lambda_{b,k}- A_\Lambda\frac{\Gamma(k-\alpha)\Gamma(b+1)}{\Gamma(k+1)\Gamma(b+1-\alpha)}\right| \\
                                                                                                            &\qquad\qquad+ A_\Lambda\frac{\Gamma(k-\alpha)}{\Gamma(k+1)}\left| b^{-\alpha} \frac{\Gamma(b+1)}{\Gamma(b+1-\alpha)} -1\right | \\
                                                                                                            & \leq \eps A_\Lambda b^{-\alpha}\frac{\Gamma(k-\alpha)\Gamma(b+1)}{\Gamma(k+1)\Gamma(b+1-\alpha)} + b^{-\alpha}p^{b-k-1}(\Lambda[0,1]+ A_\Lambda p^{1-\alpha})\\
                                                                                                            &\qquad\qquad+ CA_\Lambda \frac{\Gamma(k-\alpha)}{\Gamma(k+1)}b^{-1}\\
    & \leq C_1 \eps \frac{\Gamma(k-\alpha)}{\Gamma(k+1)} + C_2 b^{-\alpha} p^{b-k-\alpha} + C_3 b^{-1}\frac{\Gamma(k-\alpha)}{\Gamma(k+1)}.
  \end{align*}
  for some constants $C_1,C_2,C_3>0$, where we have used the fact that $p^{b-k-1}(\Lambda[0,1]+ A_\Lambda p^{1-\alpha}) \leq C_2 p^{b-k-2}$.
\end{proof}

Lemma~\ref{lemma:zeta_convergence} immediately implies the following lemma.

\begin{lemma}\label{lemma:lambda_unif}
  For every $\alpha\in (0,2)$ and $k \in \N$ fixed,
  \[
    \lim_{b \to \infty}\max_{x \in \{2/b,\dots,b/b\}} \left|b^{k-\alpha} \lambda_{bx,k} -A_\Lambda\frac{\Gamma(k-\alpha)}{k}x^{\alpha-k} \right|=0.
  \]
\end{lemma}

Next we show a result about the total rate of coalescence.
For this let
\[
  \lambda_b:=\sum_{k=2}^b \binom{b}{k}\lambda_{b,k}
\]
be the total rate of coalescence when there are $b$ blocks present.

\begin{lemma}\label{lemma:total_rates}
  For $\alpha\in (0,2)$,
  \[
    \lim_{b \to \infty}\max_{x \in \{2/b,\dots,b/b\}} \left|\frac{1}{b^\alpha}\lambda_{bx} -A_\Lambda\frac{\Gamma(2-\alpha)}{\alpha}x^{\alpha} \right|=0.
  \]
\end{lemma}
\begin{proof}
  It is easy to verify by induction that for each $b \geq 2$,
  \[
    \sum_{k=2}^b\frac{\Gamma(k-\alpha)}{\Gamma(k+1)} = \frac{\Gamma(2-\alpha)}{\alpha} - \frac{\Gamma(b+1-\alpha)}{\alpha \Gamma(b+1)}.
  \]
  Thus by Lemma~\ref{lemma:zeta_convergence}, for any $\eps>0$,
  \begin{align}\label{eq:bound_me_over_x}
    \left|\frac{1}{b^\alpha}\sum_{k=2}^{bx} \binom{bx}{k}\lambda_{bx,k} -A_\Lambda\frac{\Gamma(2-\alpha)}{\alpha}x^{\alpha} \right| &\leq x^\alpha\sum_{k=2}^{bx} \left| \binom{bx}{k}\frac{\lambda_{bx,k}}{(bx)^\alpha} - A_\Lambda \frac{\Gamma(k-\alpha)}{\Gamma(k+1)} \right| + x^\alpha\frac{\Gamma(b+1-\alpha)}{\alpha \Gamma(b+1)}  \nonumber \\
                                                                                                                                    &\leq Cx^\alpha(\eps + x^{-1}b^{-1})\frac{\Gamma(2-\alpha)}{\alpha}  + x^\alpha\frac{\Gamma(bx+1-\alpha)}{\alpha \Gamma(bx+1)} + C\frac{\sum_{k=2}^{bx} p^{bx-k-2}}{b^\alpha}\nonumber\\
                                                                                                                                    &\leq Cx^\alpha(\eps + x^{-1}b^{-1})\frac{\Gamma(2-\alpha)}{\alpha}  + x^\alpha\frac{\Gamma(bx+1-\alpha)}{\alpha \Gamma(bx+1)} + \frac{C}{(1-p)b^\alpha}.
  \end{align}

  Next we obtain uniform bounds on~\eqref{eq:bound_me_over_x} over $x \in \{2/b,\dots,b/b\}$.
  For this, notice first that $x^{\alpha-1}b^{-1} \leq b^{-1}\vee b^{\alpha-2}$.
  Next we have by Stirling's approximation that there exists a constant $C>0$ such that
  \[
    x^\alpha\frac{\Gamma(bx+1-\alpha)}{\alpha \Gamma(bx+1)}\leq C x^{\alpha}{(bx)}^{-\alpha} = C b^{-\alpha}.
  \]
  Hence in conclusion we see that there exists a constant $C'>0$ such that
  \[
    \max_{x\in \{2/b,\dots,b/b\}}\left|\frac{1}{b^\alpha}\sum_{k=2}^{bx} \binom{bx}{k}\lambda_{bx,k} -A_\Lambda\frac{\Gamma(2-\alpha)}{\alpha}x^{\alpha} \right| \leq C'(\eps +  b^{-1}\vee b^{\alpha-2})\frac{\Gamma(2-\alpha)}{\alpha} +  \frac{C'}{(1-p)b^\alpha}.
  \]
  Taking limits and using the fact that $\eps>0$ is arbitrary gives the desired result.
\end{proof}

Now we show convergence of the rate of the number of blocks involved in a merger.

\begin{lemma}\label{lemma:expected_rates}
  For $b \geq 2$ and $x \in [0,1]$ define

  \noindent\begin{minipage}{.5\textwidth}
    \[
      g(b):=\begin{cases}
        \frac{1}{b}&\text{if }\alpha\in (0,1)\\
        \frac{1}{b\log b}&\text{if } \alpha=1\\
        \frac{1}{b^\alpha}&\text{if }\alpha\in (1,2)
      \end{cases}
    \]
  \end{minipage}
  \noindent\begin{minipage}{.5\textwidth}
    \[
      \kappa(x):=A_\Lambda\times\begin{cases}
        \frac{x}{1-\alpha}&\text{if }\alpha\in (0,1)\\
        x&\text{if }\alpha=1\\
        \frac{\Gamma(2-\alpha)}{\alpha-1}x^\alpha&\text{if }\alpha\in (1,2).
      \end{cases}
    \]
  \end{minipage}

  Then for $\alpha\in (0,2)$,
  \[
    \lim_{b \rightarrow\infty}\sup_{x \in \{2/b,\dots,b/b\}} \left|g(b)\sum_{k=2}^{bx} k\binom{bx}{k}\lambda_{bx,k} -\kappa(x) \right|=0.
  \]
\end{lemma}
\begin{proof}
  Fix $\alpha\in (0,2)\backslash\{1\}$.
  One can verify by induction on $b$ that
  \[
    \sum_{k=2}^b k \frac{\Gamma(k-\alpha)}{\Gamma(k+1)}=\frac{\Gamma (2-\alpha )}{\alpha-1}-\frac{\Gamma (b-\alpha +1)b (b+1)}{(\alpha -1) \Gamma (b+2)}.
  \]
  In the case when $\alpha\in(1,2)$, the second term converges to $0$ as $b\to\infty$ and in the case when $\alpha\in(0,1)$, the second term behaves like $b^{1-\alpha}/(\alpha-1)$ as $b\to\infty$.
  Thus we see that for every $\alpha \in (0,2)\backslash\{1\}$,
  \[
    \lim_{b \to \infty}b^{\alpha}g(b)\sum_{k=2}^b k \frac{\Gamma(k-\alpha)}{\Gamma(k+1)} = \kappa(1).
  \]
  On the other hand when $\alpha=1$ we have that
  \[
    \lim_{b\to\infty}b^\alpha g(n)\sum_{k=2}^b k \frac{\Gamma(k-\alpha)}{\Gamma(k+1)}=\lim_{b\to\infty}\frac{1}{\log b}\sum_{k=2}^b \frac{1}{k-1} = 1.
  \]
  The lemma now follows from similar estimates to those in the proof of Lemma~\ref{lemma:total_rates}.
\end{proof}

We finish this section with the following result, which follows from similar computations as before and we leave the proof out.

\begin{lemma}\label{lemma:square_rate}
  For $b \geq 2$ and $x \in [0,1]$ define

  \noindent\begin{minipage}{.5\textwidth}
    \[
      \tilde g(b):=\begin{cases}
        \frac{1}{b^2}&\text{if }\alpha\in (0,1]\\
        \frac{1}{b^{2(\alpha-1)}}&\text{if }\alpha\in (1,2)
      \end{cases}
    \]
  \end{minipage}
  \noindent\begin{minipage}{.5\textwidth}
    \[
      \tilde\kappa(x):=A_\Lambda \times\begin{cases}
        \frac{x^2}{1-\alpha}&\text{if }\alpha\in (0,1)\\
        x^2&\text{if }\alpha=1\\
        \frac{x^{2(\alpha-1)}}{2-\alpha}&\text{if }\alpha\in (1,2).
      \end{cases}
    \]
  \end{minipage}

  Then for $\alpha\in (0,2)$,
  \[
    \lim_{b \rightarrow\infty}\sup_{x \in \{2/b,\dots,b/b\}} \left|\tilde g(b)\sum_{k=2}^{bx} k(k-1)\binom{bx}{k}\lambda_{bx,k} -\tilde \kappa(x)\right|=0.
  \]
\end{lemma}

\section{Convergence of the caterpillars at a given time}\label{sec:at}

Suppose now that $\Lambda$ is a finite SRV$(\alpha)$ measure with index $\alpha\in (1,2]$.
Notice that the number of $r$-caterpillars is invariant under time-changes, hence we can assume that without a loss of generality $\Lambda$ is normalised so that $A_\Lambda=1$.
We also drop $n$ from the notation and let $\Pi$ be a $\Lambda$-coalescent restricted to $\{1,\dots,n\}$.

For $r\in \N$ and $t \geq 0$ we let $Y_r(t)$ denote the number of $r$-caterpillars at time $t$, that is, $Y_r(t)$ is the number of blocks of $\Pi(t)$ that are $r$-caterpillars.
Then $Y_1(t)$ is simply the number of singletons of $\Pi(t)$ and we let $Y_0(t)$ denote the number of blocks of $\Pi(t)$.

Next, for each $r \geq 0$ and $t \geq 0$, let
\[
  X_r(t) =\begin{cases}
    \frac{1}{n}Y_r\left(t \frac{\alpha}{n^{\alpha-1}\Gamma(2-\alpha)}\right) &\text{if }\alpha\in(1,2)\\
    \frac{1}{n} Y_r(tn^{-1})&\text{if } \alpha=2
  \end{cases}
\]
and let $(\mathcal F_t:t \geq 0)$ denote the natural filtration of $(X_0,\dots,X_n)$.

Now present the main theorem of the section which we will then prove.

\begin{theorem}\label{thm:cherries_at}
  For each $T>0$ and $r \geq 0$,
  \[
    \lim_{n\to\infty}\sup_{t \leq T}|X_r(t)-x_r(t)|=0
  \]
  in almost surely, where
  \[
    x_r(t)=\begin{cases}
      (1+t)^{-\frac{1}{\alpha-1}}&\text{if }r=0\\
      (1+t)^{-\frac{\alpha}{\alpha-1}}&\text{if }r=1\\
      \frac{1}{2(r-1)!}(1+t)^{-\frac{\alpha}{\alpha-1}}\left(\frac{\alpha t}{1+t}\right)^{r-1}&\text{if }r \geq 2.
    \end{cases}
  \]
\end{theorem}

\begin{remark}
  At the time of writing this paper,~\cite{miller2016hydrodynamic} appeared, which shows Theorem~\ref{thm:cherries_at} for $r=0,1$ for Beta-distributions.
\end{remark}

We now focus on showing Theorem~\ref{thm:cherries_at}.
For a continuous time Feller process $Z=(Z_t:t\geq 0)$ adapted to a filtration $(\mathcal H_t:t\geq 0)$ define
\[
  \E[{\rm d} Z_t|\mathcal H_t]:=\lim_{\delta\downarrow 0}\frac{1}{\delta}\E[Z_{t+\delta}-Z_t|\mathcal H_t]\quad\text{ and }\quad \E[({\rm d} Z_t)^2|\mathcal H_t]:=\lim_{\delta\downarrow 0}\frac{1}{\delta}\E[(Z_{t+\delta}-Z_t)^2|\mathcal H_t]\quad t\geq 0.
\]

\begin{lemma}\label{lemma:drift}
  For $r \geq 0$ and $t \geq 0$ define
  \[
    \xi_r(t):=\begin{cases}
      -\frac{X_0(t)^\alpha}{\alpha-1} &\text{if }r=0\\
      -\frac{\alpha}{\alpha-1}X_1(t)X_0(t)^{\alpha-1}&\text{if }r=1\\
      \alpha\frac{X_1(t)^2}{2X_0(t)^{2-\alpha}}- \frac{\alpha}{\alpha-1}X_2(t)X_0(t)^{\alpha-1}&\text{if }r=2\\
      \alpha\frac{X_{r-1}(t)X_1(t)}{X_0(t)^{2-\alpha}}-\frac{\alpha}{\alpha-1}X_r(t)X_0(t)^{\alpha-1}&\text{if }r\geq 3.
    \end{cases}
  \]
  Then, almost surely
  \[
    \lim_{n\to\infty}\sup_{t \geq 0}\big|\E[{\rm d}X_r(t)|\mathcal F_t]-\xi_r(t)\big|=0.
  \]
\end{lemma}
\begin{proof}
  Notice that $Y_0(t)$ decreases by $(k-1)$ at rate $\binom{Y_0(t)}{k}\lambda_{Y_0(t),k}$.
  Thus we see that
  \[
    \E[{\rm d}X_0(t)|\mathcal F_t]=-\frac{1}{n}\sum_{k=1}^{nX_0(t)}(k-1)\binom{nX_0(t)}{k}\lambda_{nX_0(t),k}
  \]
  The result for $r=0$ now follows from Lemma~\ref{lemma:total_rates} and Lemma~\ref{lemma:expected_rates}.

  Now suppose that $r \geq 1$.
  Imagine an urn with $nX_0(t)$ many balls and for each $r \geq 1$, there are $nX_r(t)$ balls with the label $r$.
  Let us write $\chi^r_1(t),\dots,\chi^r_{nX_r(t)}(t)$ for the balls with label $r$.
  For $r \geq 1$ and $i\leq nX_r(t)$, let $A^r_i(k,t)$ be the event that when $k$ balls are chosen from the urn, uniformly at random without replacement, the ball $\chi^r_i(t)$ is chosen.
  Then at rate
  \begin{equation}\label{eq:rate}
    n^{1-\alpha}\frac{\alpha-1}{\Gamma(2-\alpha)}\binom{nX_0(t)}{k}\lambda_{nX_0(t),k}
  \end{equation}
  we have that $X_r(t)$ changes by
  \begin{equation}\label{eq:cond_change}
    \frac{1}{n}\left(\1_{\{k=2, r\geq 2\}}\sum_{i=1}^{nX_{r-1}(t)}\sum_{j=1}^{nX_1(t)}\1_{A^{r-1}_i(k,t)}\1_{A^1_j(k,t)} - \sum_{i=1}^{nX_r(t)}\1_{A^r_i(k,t)}\right).
  \end{equation}
  Indeed, at rate~\eqref{eq:rate} we select $k$ blocks uniformly without replacement, and merge these together.
  Merging together an $(r-1)$-caterpillar with a $1$-caterpillar (singleton) results in a new $r$-caterpillar and thus an increase.
  The number of $r$-caterpillars decrease whenever they are involved in the merger.

  Now, for each $k \geq 2$, $r,r' \geq 1$ and $i,j$,
  \begin{equation}\label{eq:prob_A}
    \P(A^r_i(k,t)|\mathcal F_t)=\frac{k}{nX_0(t)}\quad\text{ and }\quad\P(A^r_i(k,t);A^{r'}_j(k,t)|\mathcal F_t)=\frac{k(k-1)}{nX_0(t)(nX_0(t)-1)}\1_{\{i\neq j\text{ or }r\neq r'\}}.
  \end{equation}
  Thus by taking the conditional expectation of~\eqref{eq:cond_change}, multiplying by~\eqref{eq:rate} and summing over $k$ we get that
  \[
    \E[{\rm d}X_1|\mathcal F_t]=-\frac{X_1(t)}{X_0(t)}n^{-\alpha}\frac{\alpha-1}{\Gamma(2-\alpha)}\sum_{k=2}^{nX_0(t)}k\binom{nX_0(t)}{k}\lambda_{nX_0(t),k}
  \]
  and
  \begin{align*}
    \E[{\rm d}X_2(t)|\mathcal F_t]= & \frac{1}{2}X_1(t)(X_1(t)-1/n)n^{2-\alpha}\frac{\alpha-1}{\Gamma(2-\alpha)}\lambda_{nX_0(t),2}                                           \\ & -\frac{X_2(t)}{X_0(t)}n^{-\alpha}\frac{\alpha-1}{\Gamma(2-\alpha)}\sum_{k=2}^{nX_0(t)}k\binom{nX_0(t)}{k}\lambda_{nX_0(t),k}
  \end{align*}
  and finally for $r \geq 3$,
  \begin{align*}
    \E[{\rm d}X_r(t)|\mathcal F_t]= & X_{r-1}(t)X_1(t)n^{2-\alpha}\frac{\alpha-1}{\Gamma(2-\alpha)}\lambda_{nX_0(t),2}                                             \\
    & -\frac{X_2(t)}{X_0(t)}n^{-\alpha}\frac{\alpha-1}{\Gamma(2-\alpha)}\sum_{k=2}^{nX_0(t)}k\binom{nX_0(t)}{k}\lambda_{nX_0(t),k}
  \end{align*}
  The result now follows from applying Lemma~\ref{lemma:lambda_unif} and Lemma~\ref{lemma:expected_rates} when $\alpha\in (1,2)$, and direct computations when $\alpha=2$.
\end{proof}

Next we show that the infinitesimal variance converges to $0$ uniformly in $t$ and $r$.

\begin{lemma}\label{lemma:square_term}
  There exists a constant $C>0$, possibly depending on $\alpha\in (1,2]$, such that
  \[
    \sup_{t \geq 0, r \geq 0}\E[({\rm d}X_r(t))^2|\mathcal F_t] \leq C n^{3-\alpha}.
  \]
\end{lemma}
\begin{proof}
  Note that for each $t \geq 0$, $\eps>0$ and $r \geq 0$,
  \[
    \left|(X_r(t+\eps)-X_r(t))\right| \leq |X_0(t+\eps)-X_0(t)|
  \]
  since the change in the number of $r$-caterpillars is at most the change in the number of blocks.
  Hence we see that
  \begin{equation}\label{eq:square_by_zero}
    \sup_{t \geq 0, r \geq 0}\E[({\rm d}X_r(t))^2|\mathcal F_t] \leq \sup_{t \geq 0}\E[({\rm d}X_0(t))^2|\mathcal F_t].
  \end{equation}

  Now, $X_0(t)$ decreases by $(k-1)/n$ at rate given by~\eqref{eq:rate}.
  Hence by Lemma~\ref{lemma:square_rate} in the case when $\alpha\in (1,2)$, and trivially when $\alpha=2$, there exists a constant $C>0$, which is independent of $t$, such that
  \[
    \E[({\rm d}X_0(t))^2|\mathcal F_t] = \frac{1}{n^{1+\alpha}}\sum_{k=2}^{nX_0(t)} (k-1)^2 \binom{n X_0(t)}{k}\lambda_{n X_0(t), k} \leq C n^{3-\alpha}.
  \]
  Plugging this into~\eqref{eq:square_by_zero} finishes the proof.
\end{proof}

Now we begin to show Theorem~\ref{thm:cherries_at} by using the preceding two lemmas.
It is important to observe that for each $r \geq 0$, Doob--Mayer decomposition gives that
\[
  M_r(t):=X_r(t)-\int_0^t\E[{\rm d}X_r(s)|\mathcal F_s]{\rm d}s \qquad t \geq 0
\]
is a martingale with quadratic variation
\[
  [M_r(t)]_t=\int_0^t \E[({\rm d}X_r(s))^2|\mathcal F_s] {\rm d}s.
\]
We will show Theorem~\ref{thm:cherries_at} by induction on $r$.
We begin by proving the base case $r=0$.

\begin{lemma}\label{lemma:X_0}
  For each $T>0$ we have that almost surely,
  \[
    \lim_{n\to\infty}\sup_{t \leq T}\big|X_0(t)-(1+t)^{-\frac{1}{\alpha-1}}\big|=0.
  \]
\end{lemma}
\begin{proof}
  Fix $T>0$ and for $t \in [0,T]$ let $f(t)=|X_0(t)-x_0(t)|$ where $x_0(t)=(1+t)^{-\frac{1}{\alpha-1}}$.
  Now, $x_0(t)$ solves the integral equation
  \[
    x_0(t)=\int_0^t \frac{x_0(s)^\alpha}{\alpha-1}\,{\rm d}s \qquad t >0
  \]
  with the initial condition $x_0(0)=1$.
  Thus we see that
  \begin{align}\label{eq:G_0_apply_gronwall}
    f(t) & \leq |M_0(t)| + \int_0^t\left|\frac{x_0(s)^\alpha}{\alpha-1}-\E[{\rm d} X_0(s)|\mathcal F_s]\right|\,{\rm d} s                                                   \nonumber                                 \\
    & \leq |M_0(t)| + \int_0^t\left|\frac{X_0(s)^\alpha}{\alpha-1}-\E[{\rm d} X_0(s)|\mathcal F_s]\right|\,{\rm d} s + \frac{1}{\alpha-1}\int_0^t\left|x_0(t)^\alpha - X_0(t)^\alpha\right|\,{\rm d} s \nonumber \\
    & \leq |M_0(t)| + \int_0^t\left|\frac{X_0(s)^\alpha}{\alpha-1}-\E[{\rm d} X_0(s)|\mathcal F_s]\right|\,{\rm d} s + \frac{1}{\alpha-1}\int_0^t f(s)\,{\rm d} s
  \end{align}
  where in the final inequality we have used the fact that for $\alpha>1$ and $x,y \in [0,1]$, $|x^\alpha-y^\alpha|\leq 2|x-y|$.

  Using Gronwall's inequality and taking supremums we see that
  \[
    \sup_{t\leq T}|X_0(t)-x_0(t)|\leq \left(\sup_{t \leq T}|M_0(t)| + T\sup_{t \leq T}\left|\frac{X_0(s)^\alpha}{\alpha-1}-\E[{\rm d} X_0(s)|\mathcal F_s]\right|\right) e^{\frac{T}{\alpha-1}}
  \]
  Applying Doob's $L^2$-inequality and Burkholder-Davis-Grundy inequality we see that
  \[
    \lim_{n\to\infty}\sup_{t \leq T}|M_0(t)|=0
  \]
  in probability.
  Using this, together with Lemma~\ref{lemma:drift} shows convergence in probability and using bounded convergence finishes the result.
\end{proof}

Now we can show Theorem~\ref{thm:cherries_at}.

\begin{proof}[Proof of Theorem~\ref{thm:cherries_at}]
  Notice first that $x_r(t)$ given in the statement of Theorem~\ref{thm:cherries_at} satisfies
  \[
    \frac{{\rm d}}{{\rm d}t} x_r(t) = \begin{cases}
      -\frac{x_0(t)}{\alpha-1}&\text{if }r=0\\
      -\frac{\alpha}{\alpha-1}x_1(t)x_0(t)^{\alpha-1}&\text{if }r=1\\
      \alpha\frac{x_1(t)^2}{2x_0(t)^{2-\alpha}}-\frac{\alpha}{\alpha-1}x_2(t)x_0(t)^{\alpha-1}&\text{if }r=2\\
      \alpha\frac{x_{r-1}(t)x_1(t)}{x_0(t)^{2-\alpha}}-\frac{\alpha}{\alpha-1}x_r(t)x_0(t)^{\alpha-1}&\text{if }r\geq 3
    \end{cases}
  \]
  with the intial condition $x_0(0)=x_1(0)=1$ and $x_r(0)=0$ for $r \geq 2$.

  We proceed by induction on $r$.
  Lemma~\ref{lemma:X_0} shows the case $r=0$.
  Fix $T>0$ and suppose that there exists an $r\geq 0$ such that for every $r' \leq r$,
  \[
    \lim_{n\to\infty}\sup_{t \leq T}\big|X_{r'}(t)-x_{r'}(t)\big|=0
  \]
  in almost surely.

  Let us consider the case when $r \geq 2$, the other cases follow similarly.
  For $t \in [0,T]$ let $f(t)=|X_{r+1}(t)-x_{r+1}(t)|$, then similarly to~\eqref{eq:G_0_apply_gronwall},
  \begin{align}\label{eq:gronwall_again}
    f(t)\leq & |M_{r+1}(t)| + \int_0^t\left|\xi_{r+1}(s)-\E[{\rm d} X_{r+1}(s)|\mathcal F_s]\right|\,{\rm d} s + \alpha\int_0^t \left|\frac{X_{r}(s)X_1(s)}{X_0(s)^{2-\alpha}}-\frac{x_{r}(s)x_1(s)}{x_0(s)^{2-\alpha}}\right|\,{\rm d} s \nonumber \\
    & + \frac{\alpha}{\alpha-1}\int_0^t\left|X_{r+1}(s)X_0(s)^{\alpha-1}-x_{r+1}(s)x_0(s)^{\alpha-1}\right|\,{\rm d}s.
  \end{align}
  Now since $x_0(s),X_{r+1}(s)\leq 1$ we have that
  \begin{align*}
    |X_{r+1}(s)X_0(s)^{\alpha-1}-x_{r+1}(s)x_0(s)^{\alpha-1}| & \leq x_0(s)^{\alpha-1}|X_{r+1}(s)-x_{r+1}(s)| + X_{r+1}(s)|X_0(s)^{\alpha-1}-x_0(s)^{\alpha-1}| \\
    & \leq |X_{r+1}(s)-x_{r+1}(s)| + |X_0(s)^{\alpha-1}-x_0(s)^{\alpha-1}|.
  \end{align*}
  Plugging this into~\eqref{eq:gronwall_again}, applying Gronwall's inequality and taking supremums, we see that
  \begin{align*}
    \sup_{t \leq T}|X_{r+1}(t)-x_{r+1}(t)|\leq & \sup_{t \leq T}\left(|M_{r+1}(t)| + T\left|\xi_{r+1}(t)-\E[{\rm d} X_{r+1}(t)|\mathcal F_t]\right| + \frac{\alpha T}{\alpha-1}|X_0(t)^{\alpha-1}-x_0(t)^{\alpha-1}| \right. \nonumber \\
    & \qquad\qquad \left.+\alpha T \left|\frac{X_{r}(t)X_1(t)}{X_0(t)^{2-\alpha}}-\frac{x_{r}(t)x_1(t)}{x_0(t)^{2-\alpha}}\right|\right)e^{\frac{\alpha T}{\alpha-1}}.
  \end{align*}
  Applying Doob's $L^2$-inequality and Burkholder-Davis-Grundy inequality we see that
  \[
    \lim_{n\to\infty}\sup_{t \leq T}|M_{r+1}(t)|=0
  \]
  in probability.
  Using Lemma~\ref{lemma:drift} we have
  \[
    \lim_{n \to\infty}\sup_{t \leq T}\left|\xi_{r+1}(t)-\E[{\rm d} X_{r+1}(t)|\mathcal F_t]\right|=0
  \]
  in probability.
  The rest of the terms converge by the induction hypothesis and we see that
  \[
    \lim_{n\to\infty}\sup_{t \leq T}|X_{r+1}(t)-x_{r+1}(t)|=0
  \]
  in probability.
  Using bounded convergence gives that the above holds almost surely which concludes the proof.
\end{proof}

\section{Proof of Theorem~\ref{thm:cherry_conv}}\label{sec:proof}

For $t \geq 0$ and $r \in \{2,\dots\}$ let $Y^\uparrow_r(t)$ be the number of $r$-caterpillars seen up to time $t$, that is, the number of $r$-caterpillars $B\subset \{1,\dots,n\}$ such that $B$ is a block of $\Pi(s)$ for some $s\leq t$.
Notice that $t\mapsto Y^\uparrow_r(t)$ is increasing and $Y^\uparrow(\infty)$ is the total number of $r$-caterpillars.

Similar to before, for each $r \geq 0$ and $t \geq 0$, let
\[
  X^\uparrow_r(t) =\begin{cases}
    \frac{1}{n}Y_r\left(t \frac{\alpha}{n^{\alpha-1}\Gamma(2-\alpha)}\right) &\text{if }\alpha\in(1,2)\\
    \frac{1}{n} Y_r(tn^{-1})&\text{if } \alpha=2.
  \end{cases}
\]

Now, $Y^\uparrow_r(t)$ increases by $1$ whenever an $(r-1)$-caterpillar at time $t$ with a singleton which happens at rate
\[
  \lambda_{Y_0(t),2} Y_{r-1}(t)Y_1(t).
\]
Thus we see that
\[
  \E[{\rm d}X^\uparrow_r(t)|\mathcal F_t]=\begin{cases}
    n^{2-\alpha} \frac{X_1(t)(X_1(t)-1/n)}{2} \lambda_{nX_0(t),2}&\text{if }r=2\\
    n^{2-\alpha} X_{r-1}(t)X_1(t) \lambda_{nX_0(t),2}&\text{if }r\geq 3
  \end{cases}
\]
and
\[
  \E[({\rm d}X^\uparrow_r(t))^2|\mathcal F_t] = n^{-1}\E[{\rm d}X^\uparrow_r(t)|\mathcal F_t].
\]
Hence by Lemma~\ref{lemma:lambda_unif} and Theorem~\ref{thm:cherries_at}, we see that for $r=2$
\[
  \lim_{n\to\infty}\sup_{t \leq T}\left|\E[{\rm d}X^\uparrow_r(t)|\mathcal F_t]-\alpha\frac{x_1(t)^2}{2x_0(t)^{2-\alpha}}\right|=0
\]
in probability, and for $r \geq 3$,
\[
  \lim_{n\to\infty}\sup_{t \leq T}\left|\E[{\rm d}X^\uparrow_r(t)|\mathcal F_t]-\alpha\frac{x_{r-1}(t)x_1(t)}{x_0(t)^{2-\alpha}}\right|=0
\]
in probability.
Using a similar argument as in the proof of Theorem~\ref{thm:cherries_at} we see that
\[
  \lim_{n\to\infty}\sup_{t \leq T}\left|X^\uparrow_r(t)-x^\uparrow_r(t)\right|=0
\]
almost surely where $x^\uparrow_r(t)$ is the solution to
\[
  \frac{{\rm d}}{{\rm d}t}x^\uparrow_r(t)=\begin{cases}
    \alpha\frac{x_1(t)^2}{2x_0(t)^{2-\alpha}}& \text{if }r=2\\
    \alpha\frac{x_{r-1}(t)x_1(t)}{x_0(t)^{2-\alpha}}& \text{if }r\geq 3
  \end{cases}
\]
with the initial condition $x^\uparrow_r(0)=0$.
Using the explicit formula for $x_r(t)$ in Theorem~\ref{thm:cherries_at},
\begin{align*}
  x^\uparrow_r(\infty)&=\frac{\alpha}{2(r-2)!}\int_0^\infty (1+t)^{-\alpha/(\alpha-1)}\left(\frac{\alpha t}{1+t}\right)^{r-2} (1+t)^{(2-\alpha)/(\alpha-1)}(1+t)^{-\alpha/(\alpha-1)}\,{\rm d}t\\
                      &=\frac{\alpha^{r-1}}{2(r-2)!}\int_0^\infty \frac{t^{r-2}}{(1+t)^{r + \alpha/(\alpha-1)}}\,{\rm d}t\\
                      &=\frac{\alpha^{r-1}}{2(r-2)!}\int_0^1 u^{r-2}(1+u)^{\alpha/(\alpha-1)}\,{\rm d}u\\
                      &=\frac{\alpha^{r-1}}{2}\frac{\Gamma(1+\alpha/(\alpha-1))}{\Gamma(r+\alpha/(\alpha-1))}
\end{align*}
where in the third equality we have used the substitution $u=t/(1+t)$ and in the final equality we have used the definition of a Beta function.

Now recall that $\xi_r^{(n)}$ denotes the total number of $r$-caterpillars and so $\xi_r^{(n)}/n = X^\uparrow_r(\infty)$.
Since $t \mapsto X^\uparrow_r(t)$ and $t\mapsto x^\uparrow_r(t)$ are monotonic, we see that
\[
  \liminf_{n\to\infty} \frac{1}{n}\xi_r^{(n)} \geq \lim_{t\uparrow \infty} \liminf_{n\to\infty} X^\uparrow_r(t) = x^\uparrow_r(\infty).
\]

Now let $T_\eps:=\inf\{t \geq 0: X_0(t)<\eps\}$ and let $t_\eps:=\inf\{t \geq 0: (1+t)^{-\frac{1}{\alpha-1}}<\eps\}$.
By Theorem~\ref{thm:cherries_at}, for $n$ large enough, $T_\eps \leq t_\eps + \eps$ almost surely.
Further we have that $X^\uparrow_r(\infty)-X^\uparrow_r(T_\eps) \leq X_0(T_\eps)<\eps$ and so for large $n$,
\[
  X^\uparrow_r(\infty)\leq \eps + X^\uparrow_r(T_\eps)\leq \eps+X^\uparrow_r(t_\eps + \eps)
\]
where again we have used the fact that $t\mapsto X^\uparrow_r(t)$ is increasing. 
Taking limits on both sides we see that
\[
  \limsup_{n\to\infty}\frac{1}{n}\xi_r^{(n)} \leq \eps + x^\uparrow_r(t_\eps + \eps)
\]
almost surely.
Taking the limit as $\eps\downarrow 0$ shows that $\limsup_{n\to\infty}\xi_r^{(n)} \leq x^\uparrow_r(\infty)$ which concludes the proof.

\section*{Acknowledgements}
I would like to thank Andreas Kyprianou and Caroline Colijn for suggesting this problem to me, and discussing it with me at various stages.
This work was supported by EPSRC grants EP/L002442/1 and EP/P003818/1.

\printbibliography

\end{document}